\documentclass[11pt]{amsart}

\usepackage{amsmath,amsthm,amsfonts,amssymb,amscd}
\usepackage{mathrsfs, mathtools}
\usepackage[utf8]{inputenc}
\usepackage[T1]{fontenc}
\usepackage{latexsym}
\usepackage{graphicx,color}
\usepackage{verbatim}
\numberwithin{equation}{section}
\usepackage{enumerate}

\newtheorem{teo}{Theorem}[section]

\newtheorem{lema}[teo]{Lemma}
\newtheorem{cor}[teo]{Corollary}
\newtheorem{afir}[teo]{Claim}
\newtheorem{de}[teo]{Definition}

\newtheorem{pr}[teo]{Proposition}
\newtheorem{re}[teo]{Remark}

\newtheorem{bigteo}{Theorem}

\newtheorem{bigcor}[bigteo]{Corollary}

\author[]{Felipe Nobili}

\title[]{Partially hyperbolic sets with a dynamically minimal invariant lamination.}

\usepackage[pagewise]{lineno}
\begin{document}

\begin{abstract} 
We study partially hyperbolic sets of $C^1$-diffeomorphisms. For these sets there are defined
the strong stable and strong unstable laminations.
A lamination is called dynamically minimal when the orbit of each leaf 
intersects the set densely.

We prove that
partially hyperbolic sets having a dynamically minimal lamination 
have empty interior. We also study the Lebesgue measure and the spectral decomposition
of these sets.
These results can be applied to $C^1$-generic/robustly transitive attractors 
with one-dimensional center bundle. 

\end{abstract}


\maketitle 
\tableofcontents

\newpage

%!TEX root = paper II.tex

\section{Introduction}
%The study of hyperbolic sets has been in the heart of the Theory of Dynamical Systems in the last decades. These sets exhibit many \emph{robust} properties, meaning that these properties holds for any $C^1$ small perturbation of the initial system.  In the non-hyperbolic setting, robustness is a more unusual attribute to verify for the main traits of the dynamics. Many recent research is directed to the study of non-hyperbolic sets  having some "taste" of hyperbolicity, for whom we still expect  to find robust (or at least generic\footnote{A $C^1$-generic  property is one that holds for a residual ($G_{\delta}$ and $C^1$-dense) subset of  $\operatorname{Diff}^1(M)$. We usually refer to them as \emph{generic} for short.}) similar properties.

 %Since non-hyperbolic systems was shown to be abundant (there are open sets of them in  $\operatorname{Diff}^1(M)$), recent attention was directed to them, specially those that still have some taste of hyperbolicity, for whom we expect to find robust, or at least generic\footnote{A $C^1$-generic  property is one that holds for a residual ($G_{\delta}$ and $C^1$-dense) subset of  $\operatorname{Diff}^1(M)$. We usually refer to them as \emph{generic} for short.}, similar properties. 

%In this work we investigate some features concerning the size of the sets (existence of non-empty inteiror or positive Lebesgue volume) when the set... .

Hyperbolicity of a proper set imposes quite specific properties of its ``size'' and ``structure'', 
%Hyperbolicity has been shown to be a very restrictive condition regarding the size of sets supporting it,
especially when the dynamics on it is transitive.
For instance, it is well known that transitive hyperbolic proper sets have empty interior.
This is proved using the saturation principle in \cite{B40}\footnote{By \emph{saturation} we mean the saturation of a set by the leaves of the stable and unstable 
foliations, see also Definition~\ref{d.saturation}. For non-transitive sets, in \cite{B35} there is  an example  of a hyperbolic proper set with robustly non-empty interior.}. Bowen proved in \cite{B36}  that $C^2$ hyperbolic horseshoes have zero Lebesgue measure.
The proof of this result involves  bounded distortion arguments as well as the absolute continuity of the foliations,  ingredients which are not available for maps with less regularity.
Indeed, \cite{B31} provided an example of  $C^1$ hyperbolic horseshoe with positive Lebesgue measure.

Similar results were obtained for non-hyperbolic dynamics assuming a weaker form of hyperbolicity known as partial hyperbolicity. A set $\Lambda\subset M$ is \emph{partially hyperbolic}  for a diffeomorphism $f: M \to M$ if the tangent bundle $T_{\Lambda}M$ over the set $\Lambda$ has a dominated splitting into three $Df$-invariant subbundles $E^s \oplus E^c \oplus E^u$, where $E^s$ and $E^u$ are uniformly expanded by $Df$ and ${Df}^{-1}$, respectively. When $E^s$, $E^c$, and $E^u$ are all nontrivial, we speak of \emph{strongly partially hyperbolic} sets.

The results in \cite{B5} study the case when the non-wandering set $\Omega(f)$ is partially hyperbolic and has non-empty interior.
Recall that $C^1$-generically\footnote{We say that a property holds $C^1$-generically  if it holds for a residual ($G_{\delta}$ and $C^1$-dense) subset of the space of $C^1$ diffeomorphisms.} the set $\Omega (f)$ splits into pairwise disjoint homoclinic classes\footnote{A homoclinic class is
a (not necessarily hyperbolic) generalisation of a horseshoe:
it is a
transitive set associated to a hyperbolic periodic point $p$  defined as the closure of the transverse intersections of the stable and unstable manifolds of $p$.} which are its elementary pieces and form its spectral
decompositiom, see \cite{B34} and Definition \ref{d.sd}.
It is  proved that a strongly partially hyperbolic homoclinic class with non-empty interior is the whole manifold. Moreover, when the whole manifold is partially hyperbolic, this result holds $C^1$-openly. 
Similar results were obtained in \cite{B26}
assuming that the homoclinic class is \emph{bi-Lyapunov stable}, which is a slightly more general condition than having non-empty interior.

Finally, considering  again the Lebesgue measure of invariant transitive sets
and in the same spirit of \cite{B36}, the results in \cite{B6} extended Bowen's result 
to the partially hyperbolic setting
by showing that
sufficiently regular diffeomorphisms (of a class of differentiability bigger than one)
have no ``horseshoe-like'' partially hyperbolic sets with positive Lebesgue measure.

%In this work we deal with  \emph{partially hyperbolic} sets, which means that for a diffeomorphism $f:M\to M$, the tangent bundle $T_{\Lambda}M$ over the set $\Lambda\subset M$ splits into three $Df$-invariant subbundles $E^s \oplus E^c \oplus E^u$  such that $E^s$ and $E^u$ are uniformly expanded by $Df$ and ${Df}^{-1}$, respectively. When all $E^s$, $E^c$, and $E^u$ are non-trivial, we speak of \emph{strongly partially hyperbolic} sets. 

In this work we deal with partially hyperbolic transitive sets $\Lambda$ of $C^1$-diffeo\-mor\-phisms. 
We provide sufficient conditions guaranteing that these sets have empty interior or zero Lebesgue measure. A key feature in this setting is the existence of invariant dynamically defined \emph{laminations} integrating the bundles  $E^s$ and $E^u$, that we denote by  $\mathcal{F}^s$ and $\mathcal{F}^u$, respectively.  When 
for each leaf of the lamination its
orbit %(actually, we require only a fixed number of iterates)
has a  dense intersection with $\Lambda$, the lamination is said to be  \emph{dynamically minimal}  (see Definition \ref{def.minimal}).  In this case, we say that $\Lambda$ is an \emph{$s$-minimal} or \emph{$u$-minimal set}, according to which lamination ($\mathcal{F}^s$ or $\mathcal{F}^u$) is dynamically minimal. In \cite{p1} we prove that there is a wide class of systems verifying this property: robustly/generically transitive attractors with one-dimensional center bundle (see also \cite{B2,B10} for previous results in this direction).

 Our main result (Theorem~\ref{Tlzo}) claims that  $u$- and $s$-minimal proper sets have empty interior.
 Assuming that the central bundle is one-dimensional we prove that, $C^1$-generically,  $s$-minimal proper attractors 
have zero Lebesgue measure (see Theorem~\ref{TF}).

Another motivation of this paper concerns the spectral decomposition results for sets containing the relevant
part of the dynamics (limit, non-wandering, chain-recurrent sets, etc.). In the classical hyperbolic case, this decomposition consists of 
finitely many sets, called \emph{basic pieces}, which each is a homoclinic class, see \cite{B24}.
 Specially important sets
in this decomposition are the attractors and the repellers,
which are persistent and robustly transitive and whose
basins
form an open and dense subset of the ambient space.  There are some non-hyperbolic counterparts for this decomposition based on Conley's theory (see \cite{B34,B30}). More recently, \cite{B21} states a $C^1$-generic  spectral decomposition theorem  for chain-transitive locally maximal sets.  Here we prove a spectral decomposition theorem for $s$- and $u$-minimal homoclinic classes, see  Theorems~\ref{TG} and \ref{TH}.

\subsection{Statement of the results}

The precise definitions and notations involved in the results in this section can be found in Section~\ref{pre}.

\begin{bigteo} \label{Tlzo}
Every 
$s$- or $u$-minimal proper set has empty interior. 
\end{bigteo}

From Theorem B in \cite{p1}  (see also item (2) of  Proposition \ref{unif2} in this paper), we get 
immediately
the following corollary.

\begin{bigcor} \label{Clzo} A
 $C^1$-generic robustly transitive partially hyperbolic proper attractor with one-dimensional center bundle has  robustly empty interior.
\end{bigcor}

%Further related results on elementary pieces of dynamics with non-empty interior  can be found in  \cite{B5} and \cite{B26}.

In the next statement, $\Lambda_f(U)$ denotes the maximal invariant set of $f$ in the open set $U$. 
%See section \ref{pre} for more details and precise definitions.

\begin{bigteo} \label{TF} For a generic $f \in \operatorname{Diff}^1(M)$,  let $\Lambda_f(U)$ be a partially hyperbolic $s$-minimal proper attractor with one-dimensional center bundle. Then there are a neighborhood $\mathcal{U}$ of $f$,  an open and dense subset $\mathcal{V} \subset \mathcal{U}$, and a residual subset $\mathcal{W}$ of $\mathcal{U}$ such that: 

\begin{enumerate} 
\item $\Lambda_g(U)$ has empty interior for all $g \in \mathcal{V}$.
\item $\Lambda_g(U)$ has zero Lebesgue measure for all $g \in \mathcal{W}$.\end{enumerate}
\smallskip

Moreover, the set  $\mathcal{W}$ contains every $C^{1+\alpha}$ diffeomorfism in $\mathcal{V}$, for every $\alpha > 0$.
\end{bigteo}

Observe that item (1) of Theorem C is stronger than Corollary B, as we get robustly empty interior even if the attractor is not robustly transitive. Unfortunately, this is only obtained for the $s$-minimal case.

%Recently in \cite{B21}, they proved a spectral decomposition theorem for a more general class of sets: chain-transitive locally maximal sets. Their theorem holds in a residual subset of $\operatorname{Diff}^1(M)$. Here we prove a spectral decomposition for $s$- and $u$-minimal homoclinic classes that do not rely on genericity.
Finally, we state a spectral decomposition theorem  for $s$- and $u$-minimal homoclinic classes. Here the term \emph{minimal constant} stands for  the smallest number $d$ verifying the definition of a dynamically minimal lamination (see Definition \ref{def.minimal}). We denote by $H(p,f)$ the homoclinic class of the hyperbolic periodic point $p$ and by $\operatorname{index}(p)$  the dimension of the stable manifold of $p$.

\begin{bigteo} \label{TG}  Let $\Lambda=H(p,f)$ be an $s$-minimal (resp. $u$-minimal) isolated partially hyperbolic homoclinic class with minimal constant $d$  and $\operatorname{index}(p)= dim(E^ s)$ (resp. $\operatorname{index}(p)= dim (E^s) + dim(E^c)$). Then $\Lambda$ admits a unique spectral decomposition with exactly $d$ components.  
 \end{bigteo}

%The notation $\operatorname{index}(p)=d^s$ is stablished in the next Section. It means that the stable index of $p$ equals the  dimention of the stable bundle in the partial hyperbolic splitting of the set.\\

As a consequence of Theorem B in \cite{p1}, we obtain a robust spectral decomposition for robustly transitive attractors, meaning that every $g$ in a small neighborhood of $f$ has a spectral decomposition whose pieces are the continuations of the pieces in the spectral decomposition of $\Lambda_f$.

\begin{bigteo} \label{TH}
There is a residual subset $\mathcal{R}$ of $\operatorname{Diff}^1(M)$ satisfying the following. For every $f \in \mathcal{R}$ and $U \subset M$, if $\Lambda_f(U)$ is a partially hyperbolic robustly transitive attractor with one-dimensional center bundle, then $\Lambda_f(U)$ has a robust spectral decomposition.
\end{bigteo}

 This paper is organized as follows. In Section 2 we give the basic definitions, terminology, and state some results we use along the paper.  Theorem A is proved in subsection 3.1, Thorem C is proved in section 3.2, and  Theorems D and E are proved in section 4.

\section{Preliminaries} \label{pre}
%\subsection{The Basics}\
\medskip 

Let $M$ be a Riemannian compact manifold without boundary and, for $r \geq 1$, let $\operatorname{Diff}^{r}(M)$ be the space of $C^r$ diffeomorphisms from $M$ to itself endowed with the $C^r$-topology.

Given $f \in \operatorname{Diff}^{1}(M)$ and an open subset $U$ of $M$, we define the maximal $f$-invariant set of $f$ in $U$ by 
$$
\Lambda_f(U):=\bigcap_{n \in \mathbb{Z}} f^n(U).
$$ 

When a compact set $\Lambda$ is the maximal $f$-invariant set of some open set $U \subset M$, we say that $\Lambda$ is an \emph{isolated} set. Isolated sets vary upper semicontinuously. By an abuse of terminology, we call the set $\Lambda_g(U)$ the \emph{continuation} of the set $\Lambda_f(U)$ when $g$ varies in a small neighborhood of~$f$.

A special kind of isolated set are \emph{attractors}. We say that a set  $\Lambda$ is an attractor if there is an open set $U \subset M$ such that $\Lambda = \bigcap_{n \in \mathbb{N}} f^n(U)$ and $f(\overline{U}) \subset U$. Observe that  $M$ itself is an attractor (by taking $U =M$). The interesting case is when  $\Lambda \ne M$, when  $\Lambda$  is called a  \emph{proper} attractor.

In this work we study isolated sets with highly recurrent dynamics. We say that a set  $\Lambda$ is \emph{transitive} if there is $ x \in \Lambda$ such that its forward orbit $\mathcal{O}^+_f(x)$ is dense in $\Lambda$. In our setting, this is equivalent to the following property: For any pair $V_1,V_2$ of (relative) non-empty open sets of $\Lambda$, there is $n \in \mathbb{Z}$ such that $f^n(V_1) \cap V_2 \not= \emptyset$. A stronger recurrence property is the \emph{mixing} property: For any pair $V_1,V_2$ of (relative) open sets of $\Lambda$, there is $n \in \mathbb{N}$ such that $f^m(V_1) \cap V_2 \not= \emptyset$ for all $m \geq n$.\smallskip

We speak of a \emph{robustly} transitive set $\Lambda = \Lambda_f(U)$ when $\Lambda$ is transitive and the transitivity is also verified for the continuations $\Lambda_g(U)$ of every $g$ in a small neighborhood $\mathcal{U}$ of $f$.  If the transitivity is verified only in a residual subset of $\mathcal{U}$, then we say that $\Lambda$ is a \emph{generically transitive set}. \smallskip

In our context the isolated sets $\Lambda$ are always assumed to be partially hyperbolic with $E^s \oplus E^c \oplus E^u$ denoting the partially hyperbolic splitting of  $T_{\Lambda}M$. The values of $dim(E^s)$, $dim(E^c)$ and $dim(E^u)$, are designated by $d^s, d^c$, and $d^u$, respectively. We also assume that $\Lambda$ is robustly non-hyperbolic, meaning that $E^c$ does not have uniform contraction nor expansion in a robust way. We also require that none of the three bundles are trivial, in which case the set is strongly partially hyperbolic. See  Appendix B of \cite{B3} for a list of elementary properties and a more complete view on this topic.  
\smallskip

Partial hyperbolicity leads to the existence of dynamically defined immersed submanifolds $\mathcal{F}^s(x)$ and $\mathcal{F}^u(x)$,  through each point $x$ in the set, tangent to the stable and unstable subbundles, respectively. The set of such submanifolds are known as the stable and unstable \emph{lamination} of the set and are denoted by $\mathcal{F}^s$ and $\mathcal{F}^u$, respectively.  We direct the reader to section 3 of \cite{p1}, where the precise definition and  main properties of these laminations are provided. 
\smallskip

When dealing with perturbations of a diffeomorphism, as in the case of the continuations of isolated sets, we need to specify in these notations which diffeomorphism we are referring to. So, let $\Lambda_f(U)$ be an isolated partially hyperbolic set and  $\mathcal{U}$ be a neighborhood  of $f$ such that, for every $g \in \mathcal{U}$, the set $\Lambda_g(U)$ is partially hyperbolic with the same  bundles dimensions. We denote by $\mathcal{F}^{s}(g)$ and by $\mathcal{F}^{s}(x,g)$, respectively, the strong stable lamination of $\Lambda_g(U)$ (with respect to the partial hyperbolicity of $g$) and the leaf of this foliation that contains $x$.  Similarly,  given a hyperbolic periodic point $x \in \Lambda_g(U)$ and $\varepsilon>0$, we denote by $W^s_{\varepsilon}(x, g)$ and $W^s(x,g)$ the local stable manifold (of size $\varepsilon$) and the global stable manifolds of $x$, respectively. The union of all local or all global stable manifolds along the orbit of $x$ is denoted by $W^s_{\varepsilon}(\mathcal{O}_g(x),g)$ and $W^s(\mathcal{O}_g(x),g)$, respectively. Similarly, fixed $r>0$, we denote by $\mathcal{F}^{s}_r (x)$ the open ball of radius $r$ centered at $x$, relative to the induced distance on $\mathcal{F}^{s}(x)$. When there is no risk of misunderstanding, we simplify these notation by omitting the diffeomorphism,  as $\mathcal{F}^s(x)$ for $\mathcal{F}^s(x, f)$, $W^s(x)$ for $W^s(x, f)$, and $W^s_{\varepsilon}(\mathcal{O}_g(x))$ for $W^s_{\varepsilon}(\mathcal{O}_g(x),g)$.
\smallskip

Similar notations are considered for the unstable foliation and  manifold.

\begin{de} \label{d.saturation} \em{The \emph{saturation} of a set $K$ by a lamination $\mathcal{F}$ is the set consisting of the union of all the leaves  passing through some point of $K$. A set $K$ is \emph{saturated} by $\mathcal{F}$ if the saturation of $K$ equals $K$ (i.e, for every $x \in K$ we have $\mathcal{F}(x) \subset K$). }
\end{de}

\begin{re} \label{r.ss+1} \em{Let $\Lambda$ be a partially hyperbolic set. For every hyperbolic periodic point $p \in \Lambda$, the \emph{index} of $p$ is the dimension of $W^s(p)$ as a submanifold and is denoted by $\operatorname{index}(p)$. Since $\mathcal{F}^{s}(p)$ is a subset of  $W^s(p)$,  we have $\operatorname{index}(p) \geq d^s$. Analogously,  the strong unstable leaf of $p$ is a subset of $W^u(p)$ so $d^s+d^c+d^u - \operatorname{index}(p) \geq  d^u$. In particular, when the central bundle is one-dimensional ($d^c=1$), the index of a hyperbolic periodic point $p$ is either $d^s$ or $d^s+1$.}
\end{re}

Following \cite{p1}, given a diffeomorphism $f$ and an isolated set $\Lambda = \Lambda_f(U)$, we define the concept of \emph{compatible neighbourhood} of $f$, where the continuations of $\Lambda_f(U)$  share it main properties. 

\begin{de} \label{d.compatible}
\em{Let $\Lambda$ be an isolated set of a diffeomorphism $f \in \operatorname{Diff}^{1}(M)$ and $U \subset M$ an isolated block of $\Lambda$. We call a neighborhood $\mathcal{U}$ of $f$ a \emph{compatible} neighborhood (with respect to $U$) if $\mathcal{U}$ is sufficiently small so that, for all $g \in \mathcal{U}$: 

\begin{itemize}

\item the set $\Lambda_g(U)$ is isolated;

\item if $\Lambda_f(U)$ is an attractor of $f$, then $\Lambda_g(U)$ is an attractor of $g$;

\item if $\Lambda_f(U)$ is a partially hyperbolic set then $\Lambda_g(U)$ is a partially hyperbolic set of $g$, with the same bundles dimensions;

\item if $\Lambda_f(U)$ is a generically (resp. robustly) transitive set of $f$, then $\Lambda_g(U)$ is a generically (resp. robustly) transitive set of $g$.

\end{itemize} }

\end{de}

\subsection{Generic Isolated Sets and Attractors}\
\medskip 

In this section we gather some useful results that we invoke along our proofs. They were stablished in \cite{B4, B34, B12, p1, B16}. For convenience, we restate them here in a compact form.

%Along the proofs in this paper we make use of some results stated in \cite{p1}. For convenience, we restate them in the following list of propositions. For more details about them, see Sections  4, 5 and 9 of \cite{p1} and the original papers indicated in brackets. They correspond to the following statements in \cite{p1}, respectively:  Proposition 4.6, Proposition 4.8, Theorem 5.1, Lemma 9.4, Theorem A, Theorem B, and Corollary 4.9. \\

\begin{pr} \label{unif} There is a residual subset $\mathcal{R}$ of $\operatorname{Diff}^{1}(M)$ such that, for every $f \in \mathcal{R}$ and every isolated set $\Lambda_f(U)$, it hold:\\
\begin{enumerate}

\item if $\Lambda_f(U)$ is a transitive attractor, then there is a neighborhood $\mathcal{U}$ of $f$ such that, 
for every $g \in \mathcal{R} \cap \mathcal{U}$, the set $\Lambda_g(U)$ is a transitive attractor. \smallskip

\item  if $\Lambda_f(U)$ is non-hyperbolic, then it 
contains a pair of (hyperbolic) saddles of different indices.\medskip

\item if $\Lambda_f(U)$ is 
a transitive isolated set of $f$ that is  
partially hyperbolic with one-dimensional center bundle, then
for every pair of hyperbolic periodic points 
$p,q \in \Lambda_f(U)$ with indices $d^s$ and $d^s+1$, 
respectively,
there is an open set $\mathcal{V}_{p,q} \subset \operatorname{Diff}^{1}(M)$, 
with $f \in \overline{\mathcal{V}_{p,q}}$, satisfying:

\noindent $W^s(\mathcal{O}_g(q_g)) \subset \overline{W^s(\mathcal{O}_g(p_g))}\,$ and $\, W^u(\mathcal{O}_g(p_g)) \subset \overline{W^u(\mathcal{O}_g(q_g))}$  for every $g \in \mathcal{V}_{p,q}$. Moreover, if $\Lambda_f(U)$ is robustly transitive, then $\Lambda_g(U) \subset H(p_g,g)$.\medskip

\item if\, $\Gamma =H(p,f) $ is a partially hyperbolic homoclinic class,  then there is an extension of the partially hyperbolic splitting on $\Gamma$ to a continuous splitting on a compact neighborhood $W$ of $\Gamma$ such that it is invariant in the following sense: for every $x \in W$ with $f(x) \in W$, we have that $Df_x (E^i (x))= E^i(f(x)) \mbox{, for any}\: \;i \in\{s,c,u\}$. \medskip

\item if $\Lambda_f(U)$ is an $s$-minimal partially hyperbolic set with one-dimensional center bundle and
 $\mathcal{U}$ is a compatible neighborhood of $f$, then
for every hyperbolic periodic point $p\in \Lambda_f(U)$, 
there is an open set $\mathcal{W}_p \subset \mathcal{U}$, 
with $f \in \overline{\mathcal{W}_p}$, 
such that $H(p_g,g) \subset \overline{\mathcal{O}^-_g(D)}$  for  every strong stable disk $D$ centered at some point $x\in \Lambda_g(U)$ and every $g \in \mathcal{W}_p$. Moreover, if $\operatorname{index}(p)=d^s$, then $\mathcal{W}_p$ is a neighborhood of $f$.\medskip

\end{enumerate}
\end{pr}

Item (1) is theorem B of [1]; item (2) is due to Mane in the proof of the Ergodic Closing
Lemma \cite{B16};  item (3) is Proposition 4.8 in \cite{p1}; item (4) is Theorem 5.1 in \cite{p1} (which is a combination of Theorem 7 in \cite{B12}  and Remark 1.10 in \cite{B34}); and item (5) is Lemma 9.4 in \cite{p1}.

In the rest of this paper,  $\mathcal{R}$ always refers to the residual subset in Proposition \ref{unif}.

Fixed an open set $U \subset M$, denote by $\mathrm{RTPHA}_1(U)$ (resp. $\mathrm{GTPHA}_1(U)$) the subset of $\operatorname{Diff}^{1}(M)$ of diffeomorphisms $f$ for which the maximal $f$-invariant subset $\Lambda_f(U)$ of $U$ is a robustly (resp. generically) transitive attractor that is robustly non-hyperbolic and partially hyperbolic with one-dimensional center bundle. Observe that $\mathrm{RTPHA}_1(U)$ is an open subset of $\operatorname{Diff}^1(M)$, and that $\mathrm{GTPHA}_1(U)$ is locally residual in $\operatorname{Diff}^1(M)$. \\

Next proposition summarises Theorem A, Theorem B,  and Corollary 4.9 in \cite{p1}.

\begin{pr}[\cite{p1}] \label{unif2} For every open subset $U \subset M$, there is a residual subset $\mathcal{A}$ of $\mathrm{GTPHA}_1(U)$ and an open and dense subset $\mathcal{B}$ of $\mathrm{RTPHA}_1(U)$ such that:
\begin{enumerate}
\item for every $g \in \mathcal{A}$, the set $\Lambda_g(U)$ is either generically $s$-minimal or generically $u$-minimal.
\item for every $g \in \mathcal{B}$, the attractor $\Lambda_g(U)$ is either robustly $s$-minimal or robustly $u$-minimal. Moreover, $\Lambda_g(U)$ is a homoclinic class and depends continuously on $g\in \mathcal{B}$. 

\end{enumerate}
\end{pr}

%In what follows, we fix the notation $\mathcal{R} = \mathcal{R}_0 \cap \mathcal{R}_1 \cap \mathcal{R}_2 \cap \mathcal{R}_3$ and assume that the isolating block $U$ of an attractor $\Lambda_f(U)$ is always endowed with an extended partially hyperbolic splitting. 

\smallskip

\subsection{Lebesgue Measure and Genericity} \
\medskip

In what follows we consider the manifold $M$ endowed with a Lebesgue measure
$m$. We see how Lebesgue measure  behaves for the perturbations of an isolated set.  Observe that every isolated set $\Lambda_f(U)$ is $m$-measurable, as it is a contable intersection of open sets.\\ 

\begin{lema}\label{>} 
Let $f$ be a diffeomorphism in $\operatorname{Diff}^{1}(M)$, $\Lambda_f(U)$ be an isolated set, and $\mathcal{U}$ be a compatible neighborhood of $f$ with respect to $\Lambda_f(U)$. The map  $ \varphi \colon \mathcal{U}\to \mathbb{R}$ defined  by $\ \varphi (g)= m(\Lambda_g(U))$ is upper semicontinuous. Consequently, the set of continuity points of the map $\varphi$
is a residual subset of~$\mathcal{U}$. 
\end{lema}

\begin{proof}
Fix $g \in \mathcal{U}$ and consider the nested sequence of open sets $\Lambda(g,k) \vcentcolon= \bigcap_{n=-k}^k g^n(U).$ Clearly, $\Lambda(g,k) \searrow \Lambda_g(U)$ as $k \to \infty$. Since $m$ is a regular measure, we obtain  $\displaystyle \lim_{k \to \infty} m(\Lambda(g,k)) = m(\Lambda_g
(U))$. Thus, fixed $\varepsilon >0$, 
there is $N~=~N(g, \varepsilon) \in \mathbb{N}$ such that 
$$
m(\Lambda(g,k)) 
<m(\Lambda_g(U)) + \varepsilon = \varphi(g) + \varepsilon, \quad \mbox{for all $k \geq N$.}
$$ 
Note that there is $N_0 \in \mathbb{N}$ such that the closure of $\Lambda(g,N+N_0)$ is contained in the open set $\Lambda(g,N)$. 
Then, for every $h$ sufficiently close to $g$, it holds that 
$\Lambda(h, N+N_0) \subset \Lambda(g,N)$. Hence, 
$$
m( \Lambda_h (U)) \leq 
m(\Lambda(h,N+N_0)) \leq 
m(\Lambda(g,N)) \leq
m(\Lambda_g(U)) + \varepsilon.
$$
This means that 
$\varphi(h)\le \varphi(g)+\varepsilon$,
implying the lemma.
\end{proof}

By an standard result of topology, we get the following consequence.

\begin{cor}\label{gen}
Under the hypotheses and with the notation of Lemma \ref{>}, if there is a dense subset 
$\mathcal{W}$ of $\mathcal{U}$ such that $\varphi(g)=0$ for all $g \in \mathcal{W}$, 
then there is a residual subset $\mathcal{G}$ of $\mathcal{U}$ such that
$\varphi(g)=0$ for all $g \in \mathcal{G}$.
\end{cor}

%\begin{proof}
%By the semicontinuity in Lemma~\ref{>},  
%for each $n\in \mathbb{N}$ the set
%$\mathcal{Z}_n = \{ g \in \mathcal{U} \ | \ \varphi (g)<1/n\}$ is an open subset of $\mathcal{U}$. Since $\varphi(g) =0$ for all $g$ in a dense subset of $\mathcal{U}$, the sets $\mathcal{Z}_n$ is also dense in $\mathcal{U}$.  Now the corollary follows by setting $\mathcal{G}= \bigcap_{n \in \mathbb{N}} \mathcal{Z}_n$.
%\end{proof}

\begin{re} 
{\em{Lemma~\ref{>} and Corollary~\ref{gen}
 hold for attractors, as any attractor is an isolated set.}} 
\end{re}

\bigskip 

\section{Dynamically Minimal Laminations}\
\label{sminimality}

\subsection{$u$- and $s$-minimal sets}\

\medskip

%The results in this section do not require the set to be an attractor neither the central bundle to be one dimensional. 

\smallskip

For notational simplicity, given a strongly partially hyperbolic set $\Lambda$  we adopt the following notation. $$\mathcal{F}^{s}_{\Lambda}(x) = \mathcal{F}^{s}(x) \cap \Lambda  \quad \mbox{and} \quad \mathcal{F}^{u}_{\Lambda}(x) = \mathcal{F}^{u}(x) \cap \Lambda. $$

\begin{de}[dynamically minimal lamination] \label{def.minimal} \em{Let $\Lambda$ be a  partially hyperbolic set of a diffeomorphism $f$ with nontrivial stable bundle $E^s$. We say that the lamination $\mathcal{F}^{s}$ is \emph{dynamically minimal} (or $\Lambda$ is an $s$-minimal set) if there is $d \in \mathbb{N}$ such that,
for all $x \in \Lambda$, it holds that
$$
\bigcup_{i=1}^{d} \overline{\mathcal{F}_{\Lambda}^{s}(f^i(x))} = \Lambda.
$$

When $\Lambda=\Lambda_f(U)$ is an isolated set, $\Lambda$ is a \emph{robustly $s$-minimal set} if $\Lambda_g(U)$ is $s$-minimal for all $g$ in a neighborhood $\mathcal{U}$ of $f$. If $s$-minimality is verified only in a residual subset of $\mathcal{U}$, then $\Lambda_f(U)$ is called a \emph{generically $s$-minimal~set}.

The definition of $u$-minimality is analogous, considering the 
strong unstable lamination $\mathcal{F}^{u}$. }
\end{de}

The smallest natural number $d$ verifying this definition is called the \emph{minimal constant} of $\Lambda$. The reason we need such number $d$ of iterates to obtain the desired density property is that the attractor may not be a unique elementary piece. In fact, we prove in Section \ref{s.d} that the minimal constant $d$ is exactly the number of pieces in the spectral decomposition of $\Lambda$ (see Definition~\ref{d.sd}).  Moreover, when $\Lambda=M$, then $d=1$, so the definition of $u$- and $s$-minimality coincides with the definition of minimal foliation for partially hyperbolic diffeomorphisms. \ \

 The main result in this section is 
the following equivalence of Theorem A.

\begin{teo}\label{tint} 
Any $u$- or $s$-minimal set  
with non-empty interior is the whole manifold.
\end{teo} 

%This is actually the content of Theorem A, and together with Theorem A of \cite{p1} they immediatly imply Corollary B.

%For close related results about transitive sets with non-empty interior see, for instance, \cite{B5} and \cite{B6}. These results requires either generic arguments or that the partial hyperbolicity holds in the whole manifold $M$. Observe that the statement of Theorem~\ref{tint} do not require genericity, a fact that allow us to extract open properties as the following corollary.

%\begin{cor} \label{cint} Any robustly $s$-minimal (or $u$-minimal) proper set $\Lambda_f(U)$ (i.e., with $U \ne M$) has robustly empty interior.\end{cor}

 In the rest of this section, 
all the results are stated for $s$-minimal sets, though similar statements
(with similar proofs)
also hold in the $u$-minimal case. 

We start with some auxiliary lemmas and the following Remark, that  gives two well known properties of  the strong stable.
\\  
\begin{re} \label{r.bom}  \em{For every $r>0$ sufficiently small, it hold:\\
\begin{enumerate}[i)]
\item $\mathcal{F}^s(x) = \bigcup_{n\in \mathbb{N}} f^{-n}(\mathcal{F}_r^s( f^n(x)))$\\

\item There is $N \in \mathbb{N}$ such that $A_n(x) =  f^{-n.N}(\mathcal{F}_r^s( f^{n.N}(x)))$ yield a nested sequence (that is,  $A_n(x) \subset A_{n+1}(x)$ for every $n \in \mathbb{N}$).\\
\end{enumerate} }
\end{re}

Given a set $K \subset M$, we denote by $B_{\varepsilon}(K)$ the $\varepsilon$-neighborhood of $K$ relative to some fixed Riemannian metric on $M$. 

%That is, $B_{\varepsilon}(K)$ is the set of all points in $M$ whose distance to $K$ is less than~$\varepsilon$. 

\begin{lema} 
\label{s} 
Let $\Lambda$ be an $s$-minimal set of a diffeomorphism $f$ and $d$ be its minimal constant. Given any $\varepsilon >0$ and $r > 0$ sufficiently small, 
%there is a map $N: \Lambda \to \mathbb{N}$ such that, given $x \in \Lambda$ it holds that 
%$$ \Lambda \subset
%B_\varepsilon \Big(
%\bigcup_{i=1}^d f^{-k+i}(\mathcal{F}^{s}_r(x)) \cap %\Lambda \Big)  \quad \mbox{for all $k > N(x)$}.
%$$ 
%Moreover,
 there is a constant $N=N(\varepsilon, r) \in \mathbb{N}$ such that
$$
\Lambda \subset
B_\varepsilon \Big(
\bigcup_{i=1}^d f^{-k.N +i}(\mathcal{F}^{s}_r(x)) \Big)  \quad \mbox{for all $x \in \Lambda$ and $k \in \mathbb{N}$}.
$$ 

\end{lema}

\begin{proof} 
Fix $\varepsilon>0$ and $r> 0$. From $s$-minimality and Remark \ref{r.bom}, given any $y \in \Lambda$, there is $N_y \in \mathbb{N}$ such that 
$$
\Lambda \subset B_\varepsilon \Big(\bigcup_{i=1}^{d} f^{i} (f^{-N_y}(\mathcal{F}^{s}_r (f^{N_y}(y))))\Big).
$$
By the continuity of the foliation $\mathcal{F}^{s}$, there is a neighborhood $V(y)$ of $y$ such that the previous inclusion holds for all $z \in V(y) \cap \Lambda$, with $N_z=N_y$. Consider the covering $\{V(y)\}_{y \in \Lambda}$ of $\Lambda$. Since $\Lambda$ is a compact set, we may extract a finite subcovering $\{V(y_i)\}_{i=1}^m$ and constants $N_{y_i}$ such that, if 
$y\in \Lambda \cap V(y_j)$ for some $j \in \{1,\dots,m\}$, then
$$
\Lambda \subset  B_\varepsilon \Big(\bigcup_{i=1}^{d} f^{i} (f^{-N_j}(\mathcal{F}^{s}_r (f^{N_j}(y)))) \Big). 
$$

Let $N = \operatorname{LCM}(N_1,N_2, \cdots , N_m)$ be the lest commom multiple of these numbers. By item $ii)$ of Remark \ref {r.bom}, we can replace $N_j$ by any natural number $k.N$, with $k\in \mathbb{N}$, so we have 
$$\Lambda \subset B_\varepsilon \Big(\bigcup_{i=1}^{d} f^{i} (f^{-k.N}(\mathcal{F}^{s}_r (f^{k.N}(y)))) \Big), 
\quad  \mbox{for every $y \in \Lambda$ and $k\in \mathbb{N}$}. $$

Given $x \in \Lambda$ and $k \in \mathbb{N}$ we set $y= f^{-k.N}(x)$ in the above inclusion, so we obtain the  lemma. 
\end{proof}

\begin{lema} 
\label{ss} 
Let $\Lambda$ be an
 $s$-minimal set of a diffeomorphism $f$. If $\Lambda$ contains some strong stable disk, then $\Lambda$ contains the strong stable leaf of every point~in~$\Lambda$.  
\end{lema}

\begin{proof}
Let $r>0$ and $x_0 \in \Lambda$ be such that the strong stable disk $D =\mathcal{F}^{s}_r(x_0)$ is contained in $\Lambda$, and let  $y\in \Lambda$ be an accumulation point of the backward orbit of $x_0$.

% Fix $k \in \mathbb{N}$ and apply Lemma \ref{s} to $r$ and $\varepsilon_k = 1/k$, so we get that $\bigcup_{j=1}^{\infty} f^{-j}(D)$ is a $(1/k)$-dense subset of $\Lambda$. 
%As this holds for all $k\in \mathbb{N}$, this set is dense in $\Lambda$. To complete the proof, we use the following

%\begin{afir}
%Let $r>0$, $D$ and $x_0$ be as above, and $y\in \Lambda$ 
%be an accumulation point of the backward orbit of $x_0$. Then
%  $\bigcup_{i=1}^d f^{i}(\mathcal{F}^{s}(y))$ is a dense subset of $\Lambda$.
%\end{afir}

%\begin{proof}[Proof of the claim]
Fix $\delta> 0$ sufficiently small so that, by the partial hyperbolicity on $\Lambda$, there is $m_0\in \mathbb{N}$ such that for every stable disk $\mathcal{S}$ of length $\delta$ and $m\geq m_0$,  the image $f^m(\mathcal{S})$ is contained inside a stable disk of radius $r$. Hence, there is an increasing sequence $\{n_i\}_{n \in \mathbb{N}} \subset \mathbb{N}$, with $n_i \geq m_0$, such that $\lim_{i \to \infty}f^{-n_i}(x_0) = y$ and, for every $i \in \mathbb{N}$, the disk $f^{-n_i}(D)$ has inner radius bigger than $\delta$. By the continuity of the lamination, we obtain that $\mathcal{F}^s_{\delta}(y) \subset \Lambda$. For every $m\in \mathbb{N}$, the point  $f^{-m}(y)$ is also an accumulation point of the backward orbit of $x_0$, so the same argument leads to $\mathcal{F}^s_{\delta}( f^m(y)) \subset \Lambda$. Then we conclude that $f^{-m}(\mathcal{F}^s_{\delta}( f^m(y))) \subset \Lambda$ for every $m \in \mathbb{N}$, which implies that $\mathcal{F}(y) \subset \Lambda$ (see Remark \ref{r.bom}). Now $s$-minimality gives that $\bigcup_{i=1}^d f^{i}(\mathcal{F}^{s}(y))$ is a dense subset of~$\Lambda$.

 At this point, we concluded that every $z \in \Lambda$ is accumulated by an entire strong 
stable leaf $f^i(\mathcal{F}^s(y)) \subset \Lambda$, for some $i \in \{1,\cdots,d\}$. Since the strong stable lamination is continuous and $\Lambda$ is closed, we get that
$\mathcal{F}^{s}(z) \subset \Lambda$, ending the proof of this Lemma. \end{proof}

 %As this holds for every $\delta > r$, the whole leaf $\mathcal{F}^s(y)$ is contained in $\Lambda$. Now $s$-minimality implies that $\bigcup_{i=1}^d f^{i}(\mathcal{F}^{s}(y))$ is a dense subset of~$\Lambda$.

We are now ready to prove of Theorem~\ref{tint}

\begin{proof}[Proof of Theorem~\ref{tint}]  
Observe that the interior of $\Lambda$, denoted by $\operatorname{int}(\Lambda)$, is an invariant subset of $\Lambda$. 
Moreover, if $\Lambda$ has non-empty interior, then it contains some strong stable disk. By Lemma \ref{ss}, the set $\Lambda$ contains the strong stable leaf of every point in $\Lambda$. 

\smallskip

Suppose that the boundary $\partial \Lambda$ of $\Lambda$ is non-empty. 
Let $z \in \partial \Lambda$ and consider the disk $D =\mathcal{F}^{s}_r(z) \subset \Lambda$. By Lemma \ref{s}, there is 
$N \in \mathbb{N}$ such that $f^{-N}(D)$ intersects $\operatorname{int}(\Lambda)$. 
The $f$-invariance of $\operatorname{int}(\Lambda)$ implies that $D \cap \operatorname{int}(\Lambda) \not= \emptyset$. 
Now, choose some point $x$ in this intersection and an 
open neighborhood $B$ of $x$ with $B \subset 
\operatorname{int}(\Lambda)$. For each point $y\in B$ we consider
its entire strong stable leave $\mathcal{F}^{s}(y)$, that is
contained in $\Lambda$ (recall Lemma~\ref{ss}). 
By the continuity of the strong stable foliation, the set $V=\bigcup_{y\in B} \mathcal{F}^{s}(y)\subset \Lambda$
is a neighborhood of $\mathcal{F}^{s}(x)=\mathcal{F}^{s}(z)$. Thus $V$ is a neighborhood of $z$ that is contained in $\Lambda$,
contradicting the fact that $z\in \partial \Lambda$. Therefore $\partial \Lambda = \emptyset$, and consequently $\Lambda=M$. \end{proof}

%The next two lemmas concerns the behaviour of the strong leaves when $\Lambda$ has a minimal foliation. The first one shows that the invariant manifolds that ``contain'' the central direction intersect transversely any strong leaf of the minimal foliation and, in particular, contain a dense subset of $\Lambda$. The second one shows that the accumulation of one leaf at another one can be restricted to the set $\Lambda$.  

We end this section by providing two technical  results that will be necessary in Section  \ref{s.d}. 

First, let us recall that, by item (4) of Proposition \ref{unif}, the partially hyperbolic splitting of a generic partially hyperbolic homoclinic class $\Lambda$  extends to a neighborhood $U$ of $\Lambda$ in an invariant way. In addition, Lemma 5.3 and Remark 5.5 in \cite{p1} assure that the strong stable leave of any point in $\Lambda$ that approximate a hyperbolic periodic point in $\Lambda$ of index $d^s$ (the dimention of the stable bundle)  must transversally  intersect the unstable manifold of this point.  This is an important fact we are assuming during the proof of the following Lemma.

\begin{lema} \label{bom} Let $f \in \mathcal{R}$ and $\Lambda_f(U) = H(p,f)$ be an isolated $s$-minimal  partially hyperbolic homoclinic class of a hyperbolic periodic point $p$ of index $d^s$. Then, the unstable manifold of $p$ meets transversely any strong stable disk centered at a point in $\Lambda_f(U)$.
\end{lema}

\begin{proof}

Fix $x\in \Lambda_f(U)$, $r>0$ and $\delta>0$. Given $\varepsilon>0$, Lemma~\ref{s} gives $N \in\mathbb{N}$ such that $f^{-N}(\mathcal{F}_{r}^s(x))$ contain a point $y$ that is $\varepsilon/2$-close to $p$. By taking $\varepsilon$ sufficiently small, the disk $\mathcal{F}_{\delta}^s(y)$ intersect transversely $W^u_{\varepsilon}(\mathcal{O}_f(p))$. Moreover, by item $ii)$ of Remark \ref{r.bom},  $N$ can be chosen big enough so that, as $f$ contracts the stable leaves, $f^{N}(\mathcal{F}_{\delta}^s(y)) \subset \mathcal{F}_{2r}^s(x)$. This shows that $\mathcal{F}_{2r}^s(x)$ intersect transversely $W^u(\mathcal{O}_f(p))$.
By the arbitrary choice of $x \in \Lambda_f(U)$ and $r>0$, the conclusion follows. \end{proof}

%Fix $x\in \Lambda_f(U)$, $r>0$, $\varepsilon>0$, and $\delta>0$, where $U$ is an isolatin block of $\Lambda$. By Lemma~\ref{s}, if $\varepsilon>0$ is sufficiently small, there is $N \in \mathbb{N}$ 
%such that $f^{-N}(\mathcal{F}^{s}_r(x))$ is $\delta$-close to $p$. By  extending $f^{-N}(\mathcal{F}^{s}_r(x))$ to a $\delta$-neighborhood $\mathcal{S}$ inside the leaf $\mathcal{F}^{s}(f^{-N}(x))$, the disk $\mathcal{S}$ intersect transversely $W^u_{\varepsilon}(\mathcal{O}_f(p))$. This actually follows from the fact that $\mathcal{S}$ is tangent to a bundle over $U$ that is a continuous extention of $E^s$ (see Section 3.3 of \cite{p1}). 
%\smallskip
%
%Let $t > r$ be such that $f^{N}(\mathcal{S}) = \mathcal{F}^{s}_t(x)$. Thus,   
%\begin{equation} \label{e.extend}
%\mathcal{F}^{s}_{t}(x) \pitchfork f^{N}( W^u_{\varepsilon}(\mathcal{O}_f(p)) 
%\not= \emptyset.
%\end{equation}
%
%We can choose $N$ big enough so that, by the exponential contraction on the disk $\mathcal{S}$ (see item (3) of Proposition~\ref{p.foliation}), the estimation $t < 2r$ holds. Since we can choose $r$ arbitrarily small, Equation~\eqref{e.extend} and this estimation imply that $W^u( \mathcal{O}_f(p))$ meets transversely any strong stable disk centered at~$x$.

\begin{lema}  \label{l.restriction0} Let $f \in \mathcal{R}$ and $\Lambda =  H(p,f)$ be an isolated $s$-minimal partially hyperbolic set of some hyperbolic periodic point $p$ of index $d^s$. Then, for every $x, y \in \Lambda$ satisfying  $\mathcal{F}^s(x) \subset \overline{\mathcal{F}^s(y)}$ it holds that $\mathcal{F}_{\Lambda}^s(x) \subset \overline{\mathcal{F}_{\Lambda}^s(y)}$.
\end{lema}

\begin{proof} Let $z \in \mathcal{F}^s_{\Lambda}(x)$, $r>0$ and consider the disk $\mathcal{F}^s_r(z)$. By Lemma~\ref{bom}, $W^u(p)$ meets transversely $\mathcal{F}^s_r(z)$, say at the point $w$. Since     
$\mathcal{F}^s(x) \subset \overline{\mathcal{F}^s(y)}$, we also have an intersection point $\hat{w}$ of $\mathcal{F}^s(y)$ and $W^u(p)$ that can be choosen arbitrarily close to $w$. From $s$-minimality, the orbit of $\mathcal{F}^s(p)$ accumulates at $\mathcal{F}^s(y)$ and thus intersect transversely $W^u(p)$ in a sequence of points that accumulate to $\hat{w}$. This sequence of points consist of transverse homoclinic points of $p$, so $\hat{w} \in \Lambda$. As $r$ can be chosen arbitrarily small and $\hat{w}$ can be chosen arbitrarily close to $w$, we conclude that $z \in \overline{\mathcal{F}^s_{\Lambda}(y)}$. Since it holds for every $z \in  \mathcal{F}_{\Lambda}^s(x)$ we finally obtain that $\mathcal{F}_{\Lambda}^s(x) \subset \overline{\mathcal{F}_{\Lambda}^s(y)}$.\end{proof}

\bigskip

\subsection{$s$-minimal attractors}\ \label{sattractor}

\medskip

In what follows we study $s$-minimal attractors apart, with no  similar statements to the case of $u$-minimal attractors \footnote{ Recall that by taking $f^{-1}$, the attractor becomes a repellor. }.

The main result presented here is Theorem C. Before proving it, we need some intermediate results that also hold for $d^c \geq 1$.

%Item (2) of Theorem~C could be restated as the following corollary, whose proof is an immediate consequence of Theorem A* of \cite{p1}. 
%
%\begin{cor} $C^1$-generically, a transitive a partially hyperbolic attractor with positive Lebesgue measure  and one-dimensional center bundle is $u$-minimal.
%\end{cor}

In the next statements, the notation $\operatorname{Per}_{\sigma} (f_{|_{\Lambda}})$ stands for the set of hyperbolic periodic points in $\Lambda$ of index $\sigma$.

\begin{lema}\label{sss} Let $\Lambda =\Lambda_f(U)$ be a partially hyperbolic attractor that is $s$-minimal, contains some strong stable disk, and has a point $ p \in \operatorname{Per_{d
^s}(f_{|_{\Lambda}})}$. Then $\Lambda$ is the whole manifold.
\end{lema}

\begin{proof}
 By Theorem~\ref{tint}, it suffices to prove that $\Lambda$ has non-empty interior. Consider the periodic point $p \in \operatorname{Per_{d^s}(f_{|_{\Lambda}})}$. Then, for a small $\varepsilon>0$, its local unstable manifold $W^{u}_{\varepsilon}(p)$ is a $(d^u+d^c)$-dimensional
embedded manifold contained in the attractor. By Lemma \ref{ss}, the strong stable leaf of any point in $\Lambda$ is contained in $\Lambda$. 
Thus the saturation 
of $W^{u}_{\varepsilon}(p)$ by its strong stable leaves contains an open subset of $\Lambda$, so $\Lambda$ has non-empty interior.
\end{proof}

%Let $\operatorname{Diff}^{1+}(M)$ denotes the subset of $\operatorname{Diff}^{1}(M)$ of diffeomorphisms whose derivative is $\alpha$-holder for some $\alpha > 0$ (that is,  $\operatorname{Diff}^{1+}(M) = \bigcup_{\alpha > 0} \operatorname{Diff}^{1+\alpha}(M))$. 
The following  proposition is a simplified version of Corollary B in \cite{B6} for the case of partially hyperbolic attractors. 

\begin{pr}[\cite{B6}]
\label{lcompletar}
Fix $\alpha>0$ and $f \in \operatorname{Diff}^{1+\alpha }(M)$. If $\Lambda$ is a partially hyperbolic set of $f$ with $m(\Lambda) > 0$, then $\Lambda$ contain some strong stable disk and some strong unstable disk.
\end{pr}

\begin{lema} \label{J} 
Let $f \in \operatorname{Diff}^{1+\alpha}(M)$ and $\Lambda = \Lambda_f(U)$ be  partially hyperbolic attractor that is $s$-minimal. If $\operatorname{Per_{ds}(f_{|_{\Lambda}})} \ne \emptyset$ and
$m(\Lambda) > 0$, then $\Lambda$ is the whole manifold.
\end{lema}

\begin{proof} By Proposition~\ref{lcompletar} there is a strong stable disk $D$ contained in $\Lambda$. Now Lemma \ref{sss} implies the statement. \end{proof}

We are now ready to prove Theorem~C. 

\begin{proof}[Proof of theorem~C] Since $f$ is $C^1$-generic and $\Lambda_f(U)$ is $s$-minimal, we can assume that $\Lambda_f(U)$  is generically $s$-minimal (see Proposition \ref{unif2}).  Let  $\mathcal {U}$ be a compatible neighbourhood of $f$ and $\mathcal{J}_0$ be the residual subset of $\mathcal{U}$ of diffeomorphisms $g$ such that $\Lambda_g(U)$ is $s$-minimal. 

\begin{afir}
 \label{cl.estrelinha}
For every $g \in \mathcal{J}_0$, $\varepsilon > 0$, and 
every hyperbolic periodic point $a \in \Lambda_g(U) \cap \operatorname{Per}_{d^s+1}(g) $ it holds that
$$
\operatorname{int}(W_{\varepsilon}^s(a) \cap \Lambda_g(U)) = \emptyset. $$
Here the interior refers to the topology of $W_{\varepsilon}^s(a)$.
\end{afir}

\begin{proof}[Proof of the claim]
 The proof is by contradiction.
Assume that there are $\varepsilon > 0$ and $a \in \Lambda_g(U) \cap \operatorname{Per}_{d^s+1}(g) $ such that
$\operatorname{int}(W_{\varepsilon}^s(a,g) \cap \Lambda_g(U))$
contains an open ball $B$ of $W^s_{\varepsilon}(a,g)$.
By saturating $B$
 with strong unstable leaves (which are subsets of the attractor $\Lambda_g(U)$) 
 we get an open set (relative to the ambient manifold $M$)
contained in $\Lambda_g(U)$. Thus $\Lambda_g(U)$ has non-empty interior
and, by Theorem~\ref{tint} it is the whole manifold, contradicting the fact that
$\Lambda_g(U)$ is a proper attractor.
\end{proof}

Consider a diffeomorphism $f$ as in the statement of Theorem C and 
a pair of hyperbolic periodic points $p,q \in \Lambda_f(U)$  
with indices  $d^s$ and $d^{s}+1$, respectively (these points exist by item (2) of Proposition \ref{unif} and Remark~\ref{r.ss+1}). Let $\mathcal{W}_p$ and  $\mathcal{V}_{p,q}$  be the open sets given by items (3) and (5) of Proposition \ref{unif}, respectively. By shrinking $\mathcal{W}_p$ if necessary, we can assume that $\mathcal{W}_p \subset \mathcal{V}_{p,q}$, so the continuation $q_g$ of $q$ is well defined for every $g \in \mathcal{W}_p$.

\begin{afir}
The map $\phi$ given by $g \mapsto W^s_{\varepsilon}(q_g, g) \cap \Lambda_g(U)$,
defined on $\mathcal{W}_p$, is upper semicontinuous.
\end{afir}
\begin{proof}

Observe that, for every $g \in \mathcal{W}_p$, the set $\{ \mathcal{F}^u_{\varepsilon}(x) \ | \  x \in W^s_{\varepsilon}(q_g, g)\cap \Lambda_g(U)\}$ is an open subset of $\Lambda_g(U)$. Since $W^s_{\varepsilon}(p_g,g)$ varies continuously, this observation shows that an upper discontinuity of $\phi$ would imply an upper discontinuity of $\Lambda_g(U)$. However, such a discontinuity for $\Lambda_g(U)$ is not possible as attractors vary upper semicontinuously. 
\end{proof}

As a consequence of this claim, there is a residual subset  $\mathcal{J}_1 \subset \mathcal{W}_p$ 
consisting of continuity points of the map $\phi$.

By Claim~\ref{cl.estrelinha} and the definition of $\mathcal{J}_1$ 
we conclude that, for every $h \in \mathcal{J}_0 \cap \mathcal{J}_1$ (that is a subset of $\mathcal{W}_p$), 
there is a neighborhood $\mathcal{U}_h$ of $h$ such that
\begin{equation} \label{e.notsubset}
W_{\varepsilon}^s(q_g , g) \not\subset  \Lambda_g(U) \quad \mbox{for all} \quad g \in \mathcal{U}_h.
\end{equation}

The set $\mathcal{V}_p= \bigcup_{h \in \mathcal{J}_0 \cap \mathcal{J}_1} \mathcal{U}_h$  is an open and dense subset of $\mathcal{W}_p$.  

\begin{afir} \label{c.oed} For every $g \in \mathcal{V}_p$ the  attractor $\Lambda_g(U)$ does not contain any strong stable disk, and consequently it has empty interior.
\end{afir}

\begin{proof}
Suppose that there is $g \in \mathcal{V}_p$ for which $\Lambda_g(U)$ has a strong stable disk $D \subset \Lambda_g(U)$. By the invariance and closeness of $\Lambda_g(U)$, any accumulation point of the backward orbit of $D$ belongs to $\Lambda_g(U)$. By item (4) of Proposition \ref{unif}, the closure of the negative orbit of $D$ contains $H(p_g,g)$, so we conclude that $\overline{\mathcal{F}^{s}(p_g , g)} \subset \Lambda_g(U)$. Now, item (3) of Proposition \ref{unif} implies that $W^s(q_g,g) \subset \Lambda_g(U)$, contradicting Equation~\eqref{e.notsubset}. 
\end{proof}

Recall that $\mathcal{V}_p$ depends on the choice of $f \in \operatorname{Diff}^1(M)$ and, since $f \in \overline{\mathcal{W}_p}$, we also have $f \in \overline{\mathcal{V}_p}$. Hence, to obtain item (1) of Theorem~C, we apply Claim~\ref{c.oed} with respect to every diffeomorphism in $\mathcal{R} \cap \mathcal{U}$. The union of all open sets obtained in this way is the announced open and dense subset $\mathcal{V}$~of~$\mathcal{U}$.

\smallskip

Fix $\alpha >0$. To prove the second part of the theorem, observe that, if $g \in \mathcal{V} \cap \operatorname{Diff}^{1+\alpha}(M)$ is such that $m(\Lambda_g(U)) > 0$, then it contains a strong stable disk (see Proposition~\ref{lcompletar}). This contradicts Claim~\ref{c.oed}, since we have taken $g \in \mathcal{V}$. This proves that the subset of $\mathcal{U}$ for which $\Lambda_g(U)$ has zero Lebesgue measure contains  every $C^{1+\alpha}$ diffeomorphism of $\mathcal{V}$. 
\smallskip

In particular, for every $C^2$ diffeomorphisms $g$ in $\mathcal{V}$, the attractor $\Lambda_g(U)$ has zero Lebesgue measure. Since the subset of $C^2$ diffeomorphisms in $\mathcal{V}$ is $C^1$-dense in $\mathcal{V}$, Corollary~\ref{gen} implies that there is a residual (with respect to the $C^1$ topology) subset of $\mathcal{V}$ where the attractors have zero Lebesgue measure.
\end{proof}

\bigskip

\section{Spectral Decomposition} \label{s.d}
 
\medskip
 
%Recall that, in the definition of an $s$-minimal set $\Lambda$, we have fixed a natural number $d$ for which the orbit segment of length $d$ of any  strong stable leaf of $\Lambda$ intersect $\Lambda$ in a dense subset.  
 
In this section we see how $u$- and $s$-minimal homoclinic classes are decomposed into a finite number  of compact sets which are permuted   by the dynamics and verify the strong recurrence property of mixing. Moreover, the number of pieces in this decomposition is exactly  the minimal constant $d$  in Definition \ref{def.minimal}.  Let us describe it more precisely.  
 
% The results in this section also holds for $u$-minimal sets with analogous proofs. As usual, we only state and prove them for the $s$-minimal case.
%
%\smallskip
%
%Consider a transitive compact and invariant proper set $\Lambda$ of a diffeomorphism $f \in \operatorname{Diff}^{1}(M)$. The set $\Lambda$ can be written as the union of a finite number $k$ of pairwise disjoint compact subsets $\{\Lambda_i\}_{i=1}^k$ that are permuted by the action of the diffeomorphism $f$. This is clear by setting $k=1$, so the interesting question is how big can $k$ be taken.  It is not always the case that there is a maximum $k$ verifying this condition.  Nevertheless, the well known Spectral Decomposition Theorem gives an affirmative answer in the case of hyperbolic sets. 
%
%Note that, when such a maximum $k$ exists, then the transitivity of the set $\Lambda$ implies that  the permutation must be cyclic (that is, $\Lambda = \bigcup_{i=1}^k f^i (\Lambda_1)$). In some sense, we can say that each compact subset $\Lambda_i$ as above is a ``copy'' of the other ones, and they exhibit the  ``same'' dynamics for $f^k$. The fact that $k$ is the biggest number verifying this condition implies that $\Lambda_i$ is "dynamically indecomposable", that is, do not admits any further decomposition of this type. 
%
%
%A special case is when the sets $\Lambda_i$ is mixing for $f^k$, a condition that is always met when $\Lambda$ is a hyperbolic transitive set.  In this case we speak of a \emph{spectral decomposition}.  
%
%\smallskip

\begin{de}[Spectral decomposition] \label{d.sd} \em{ We say that a transitive compact invariant set $\Lambda$ admits a \emph{spectral decomposition} if there exist compact sets $\Lambda_1, \Lambda_2,..,\Lambda_k$ satisfying:
 \begin{enumerate}
 
 \item $\Lambda= \bigcup_{i=1}^k \Lambda_i$.
 
 \smallskip
 
 \item There is a cyclic permutation $\sigma: \{1,...,k\} \circlearrowleft $ such that  $f(\Lambda_i) = \Lambda_{\sigma(i)}$ for all $i \in \{1,...,k\}$. In particular, $\Lambda_i$ is periodic with period $k$.  

\smallskip

\item They are pairwise disjoint: $\Lambda_i \cap \Lambda_j =\emptyset$ for all $i \not= j$ in $\{1,...,k\}$.

%\item For every $i \in \{1,...,d\}$, $\Lambda_i$ is saturated by its strong stable leaves and, for every point  $ x \in \Lambda_i$,  $\mathcal{F}_{\Lambda}^{s}(x)$ is a dense subset of $\Lambda_i$.

\smallskip

\item For every $i \in \{1,...,k\}$, $\Lambda_i$ is topologically mixing for the map $f^{k}$.\\

 \noindent We call the sets $\Lambda_i$ the \emph{basic components} or the \emph{basic pieces} of $\Lambda$.

\end{enumerate} }
\end{de}

\begin{re} \label{r.multiple} \em{As the permutation in item (2) is cyclic, the period of any periodic point in $\Lambda$ is a multiple of the number $k$ of components of $\Lambda$.} 
\end{re}

The main results in this section are Theorem D and its robust version for robustly transitive attractors in Theorem~E. All the statements and proves in this section deal only with the $s$-minimal case. The $u$-minimal case readily follows by applying these results to the inverse map $f^{-1}$. \smallskip

To prove these theorems we start with some auxiliary lemmas. 

\begin{lema} \label{min}  Let $\Lambda=H(p,f)$ be an isolated $s$-minimal set with minimal constant $d$ and $\operatorname{index}(p) =d^s$. Let $x \in \Lambda$ and $k >1$ be such that $$\bigcup_{i=1}^{k} \overline{\mathcal{F}_{\Lambda}^{s}(f^i(x))} = \Lambda.$$
Then $k \geq d$. 
\end{lema}

\begin{proof} Fix $y \in \Lambda$. From $s$-minimality, we get that $$\bigcup_{i=1}^{d} \overline{\mathcal{F}_{\Lambda}^{s}(f^i(y))} = \Lambda.$$ 
Then there is some $m \in \{1,...,d\}$ such that $x \in \overline{\mathcal{F}^{s}_{\Lambda}(f^m(y))}$. It follows from the continuity of the foliation that $\mathcal{F}^{s}(x) \subset \overline{\mathcal{F}^{s}(f^m(y))}$ (see Proposition 5.4 of \cite{p1}). By Lemma~\ref{l.restriction0}, we get that:

$$ \Lambda = \bigcup_{i=1}^{k} \overline{\mathcal{F}_{\Lambda}^{s}(f^i(x))} \subset \bigcup_{i=1}^{k} \overline{\mathcal{F}_{\Lambda}^{s}(f^{m+i}(y))}= f^m(\bigcup_{i=1}^{k} \overline{\mathcal{F}_{\Lambda}^{s}(f^{i}(y))}) \subset \Lambda.$$

Thus $f^m(\bigcup_{i=1}^{k} \overline{\mathcal{F}_{\Lambda}^{s}(f^{i}(y))}) = \Lambda$, and consequently $\bigcup_{i=1}^{k} \overline{\mathcal{F}_{\Lambda}^{s}(f^{i}(y))} = \Lambda$. As it holds for every $y \in \Lambda$, the constant  $k$ satisfies the $s$-minimality condition. Now, the definition of minimal constant implies that $k \geq d$. 
\end{proof}

\begin{lema} \label{disj}  Let $\Lambda$ be as in Lemma~\ref{min}. For every $x \in \Lambda$ the sequence of sets $\{\overline{\mathcal{F}^{s}_{\Lambda}(f^n(x))}\}_{n=1}^d$ is pairwise disjoint.
\end{lema}

\begin{proof} The proof is by contradiction. Suppose there is  $z \in \overline{\mathcal{F}_{\Lambda}^{s}(f^i(x))} \cap \overline{\mathcal{F}_{\Lambda}^{s}(f^j(x))} $ for some $i < j$ in $\{1,...,d\}$.  By Lemma~\ref{l.restriction0}, the set $\mathcal{F}_{\Lambda}^{s}(z)$ is contained in this intersection.  Since  $\mathcal{F}_{\Lambda}^{s}(z) \subset \overline{\mathcal{F}_{\Lambda}^{s}(f^j(x))}$,   we obtain
\begin{equation} \label{e.1m}
\bigcup_{n=1}^{j-i} \overline{\mathcal{F}_{\Lambda}^{s}(f^n(z))} \subset \bigcup_{n=j+1}^{2j-i} \overline{\mathcal{F}_{\Lambda}^{s}(f^n(x))}.
\end{equation}
Similarly, since  $\mathcal{F}_{\Lambda}^{s}(z) \subset \overline{\mathcal{F}_{\Lambda}^{s}(f^i(x))}$, we have that $\mathcal{F}_{\Lambda}^{s}(f^{j-i}(z)) \subset \overline{\mathcal{F}_{\Lambda}^{s}(f^j(x))}$, and consequently we obtain
\begin{equation} \label{e.md}
\bigcup_{n=j-i+1}^{d} \overline{\mathcal{F}_{\Lambda}^{s}(f^n(z))} \subset \bigcup_{n=j+1}^{d+i} \overline{\mathcal{F}_{\Lambda}^{s}(f^n(x))}.
\end{equation}

Denoting $r = max\{2j-i , d+i\}$, $w = f^j(x)$, and putting together Equations \eqref{e.1m} and \eqref{e.md},  we conclude that

$$\Lambda = \bigcup_{n=1}^{d} \overline{\mathcal{F}_{\Lambda}^{s}(f^n(z))} \subset \bigcup_{n=j+1}^{r} \overline{\mathcal{F}_{\Lambda}^{s}(f^n(x))} = \bigcup_{n=1}^{r-j} \overline{\mathcal{F}_{\Lambda}^{s}(f^{n}(w))}.$$
This contradicts Lemma \ref{min}, since  $r - j = max\{ j-i, d-j+i\} < d$. 
\end{proof}
\smallskip
Now we are ready to prove Theorem D.

\begin{proof}[Proof of Theorem D] We have to prove items (1),(2),(3) and (4) of Definition \ref{d.sd} with $k=d$. 

\smallskip

Take some $x \in \Lambda$ and set $\Lambda_i = f^{i}(\overline{\mathcal{F}^{s}_{\Lambda}(x)})$ for $i \in \{1,\dots,d\}$. Item (1) of Definition \ref{d.sd} is an immediate consequence of $s$-minimality. 

\smallskip

For item (2), set $\sigma(i)=i+1$ for $1 \leq i < d$ and $\sigma(d)=1$. It is clear that $f(\Lambda_i) = \Lambda_{i+1} = \Lambda_{\sigma(i)}$ for all $1 \leq i <d$. So we only have to prove that $f(\Lambda_d) = \Lambda_{\sigma(d)} = \Lambda_1$.

\smallskip

Applying Lemma \ref{disj} to $x$ and $f(x)$, and the using the fact that $\Lambda$ is $s$-minimal, we have 
$$\Lambda = \bigcup_{n=1}^{d} \overline{\mathcal{F}_{\Lambda}^{s}(f^n(x))} =\bigcup_{n=2}^{d+1}\overline{\mathcal{F}_{\Lambda}^{s}(f^n(x))},$$

\noindent where both unions consist of pairwise disjoint sets. Hence, "substracting"  $\bigcup_{n=2}^{d}\mathcal{F}_{\Lambda}^{s}(f^n(x))$ in this equation, we obtain that $\overline{\mathcal{F}_{\Lambda}^{s}(f(x))}$ $=\overline{\mathcal{F}_{\Lambda}^{s}(f^{d+1}(x))}$, which means that $\Lambda_1 =f(\Lambda_d)$.

\smallskip

Item (3) is just Lemma \ref{disj}. 

For item (4), fix $i \in \{1, \dots , d\}$ and two relative open sets $A, B$ of $\Lambda_i$. Consider a hyperbolic periodic point $q \in A$ and $r>0$ such that $\mathcal{F}_r^s(q) \cap \Lambda_i \subset A$. Let $\varepsilon>0$ be such that every $\varepsilon$-dense subset of $\Lambda_i$ intersects $B$.

From $s$-minimality, there is $k \in \mathbb{N}$ sufficiently big so that $f^{-k-n.d}(\mathcal{F}_r^s(q)) $  is $\varepsilon$-dense in $\Lambda_i$ for every $n\in \mathbb{N}$.  Clearly, $k$ must be a multiple of $d$, as both $\mathcal{F}_r^s(q)$ and  $B$  belong to the same component $\Lambda_i$. Then, for some fixed $L \in \mathbb{N}$, we can write

$$f^{-d.(L+n)}(\mathcal{F}_r^s(q)) \cap B \ne \emptyset,\, \mbox{for every} \ n\in \mathbb{N}. $$

In particular, $f^{n.d} (B)\cap A \ne \emptyset$ for every $n > L$. Since we have chosen $A$ and $B$ as arbitrary relative open subsets of $\Lambda_i$,  we conclude that  $f^d$ is mixing on $\Lambda_i$.\end{proof}

%\cap B \ne \emptyset$ for every $k \geq k_0$.
%
%
% $j \in \{0, \dots , d-1\}$ such that $f^{j}(\mathcal{F}^s(p_a)) \cap B \ne \emptyset$. Since $A$ and $B$ lie in the same component $\Lambda_i$, $j$ must be zero. 
% Then, for every $r>0$ there is $k$ (a multiple of the period of $p_a$) sufficiently large so that $$f^{-k}(\mathcal{F}_r^s(p_a)) \cap B \ne \emptyset.$$ 
%
%By the cyclic permutation in item (2), we also have that 
%$$f^{-k -n.d}(\mathcal{F}_r^s(p_a)) \cap B \ne \emptyset,$$ 
%for every $n \in \mathbb{N}$. Since $k$ is a multiple of the period of $p_a$ and by Remark~\ref{r.multiple}, we can write  
%$$f^{d.(L-n)}(\mathcal{F}_r^s(p_a)) \cap B \ne \emptyset,$$
%for some $L \in \mathbb{Z}$. This means that the map $f^d$ is mixing on $\Lambda_i$.\end{proof}   

\begin{teo} \label{t.same.d} Let $\Lambda=H(p,f)$  be as in Lemma~\ref{min} for some generic $f \in \mathcal{R}$. Then there is a neighborhood $\mathcal{U}$ of $f$ such that, for every $g \in \mathcal{U}$ that is $s$-minimal, the minimal constant of $g$ is also $d$. 
\end{teo}

\begin{proof} Let $m$ be the period of the hyperbolic periodic point $p$. By Theorem D and Remark \ref{r.multiple}, there is $n \in \mathbb{N}$ such that $m= n \cdot d$. From $s$-minimality, we get that $\Lambda = \bigcup_{n=1}^{d} \overline{\mathcal{F}_{\Lambda}^{s}(f^n(p))}$. By item (2) in definition \ref{d.sd}, with $k=d$, for every $i \in \{1,\dots,d\}$ it holds that 

\begin{equation} \label{e.l1} \Lambda_i = \overline{\mathcal{F}_{\Lambda}^{s}(f^i(p))} = \overline{\mathcal{F}_{\Lambda}^{s}(f^{d+i}(p))} = \dots = \overline{\mathcal{F}_{\Lambda}^{s}(f^{(n-1)d+i}(p))}. 
\end{equation}

This equation implies that $\mathcal{F}^{s}(f^i(p))$ intersects transversally the unstable manifold of $f^{d+i}(p)$, $f^{2d+i}(p), \dots$, and $f^{(n-1)d}(p)$. Clearly, these transverse intersections occur robustly in a small neighbourhood $\mathcal{U}$ of $f$. Hence, by the $\lambda$-lemma, for every $g \in \mathcal{U}$ it holds that

\begin{equation} \label{e.l1} \overline{\mathcal{F}_{\Lambda_g}^{s}(g^i(p))} = \overline{\mathcal{F}_{\Lambda_g}^{s}(g^{d+i}(p))} = \dots = \overline{\mathcal{F}_{\Lambda_g}^{s}(g^{(n-1)d+i}(p))}. 
\end{equation}

This shows that the number of pieces in the spectral decomposition of $\Lambda_g$ for $g$ in a small neighborhood of $f$ cannot increase (is at most $d$).

\smallskip

On the other hand, the pairwise disjoint compact isolated sets $\{\Lambda_i\}_{i=1}^d$ admit upper semicontinuations for any diffeomorphism $g$ sufficiently close to $f$, and the cyclic permutation given by $f$ induces a cyclic permutation given by $g$ on these continuations. Hence the number of components of $\Lambda_g(U)$ do not decrease in a small neighborhood of $f$. 

\smallskip

As a conclusion, the spectral decomposition of $g$ has exactly $d$ components. Then $d$ must be the minimal constant of the $s$-minimality of $\Lambda_g$. \end{proof}

\begin{proof}[
Proof of Theorem E]  By item (2) of Proposition \ref{unif2},  we can assume that $f$ is either robustly $s$-minimal or robustly $u$-minimal. Without loss of generality, we admit that $f$ is robuslty $s$-minimal (with minimal constant $d$). We can also assume that $\Lambda_f(U)$ is robustly a homoclinic class, and that $\Lambda_g(U)$ vary continuously in a neighborhood of $f$ (see Corollary 4.9 in \cite{p1}). Then $\Lambda_g(U)$ consist of $d$ attractors of $f^d$ that are the continuations of the components of $\Lambda_f(U)$. By theorem~\ref{t.same.d}, the spectral decomposition of $\Lambda_g(U)$ has exactly $d$ components, so they must coincide with the continuations of the pieces in the spectral decomposition of $\Lambda_f(U)$.  
 \end{proof}
 
\section*{Acknowledgements}
This work is part of my PhD  Thesis at PUC-Rio. I am specially grateful to Lorenzo J. Díaz for his constant support and attention.  I thank Flavio Abdenur for his commitment and help during my stay at PUC.  I am also thankful to  CNPq and FAPERJ for funding me during this project.


\begin{thebibliography}{100}



\bibitem{B4} F. Abdenur. \textit{ Attractors of Generic Diffeomorphisms are Persistent}, Nonlinearity 16 (2003), no. 1, 301-311.

\bibitem{B5} F. Abdenur, C. Bonatti, and L. Díaz. \textit{Non-wandering sets with non-empty interior}, Nonlinearity 17 (2004), no. 1, 175-191. 


\bibitem{B21} F. Abdenur and S. Crovisier, \emph{Transitivity and topological mixing for $C^1$ diffeomorphisms},  Essays in mathematics and its applications, 1--16, Springer, Heidelberg, 2012.
 
\bibitem{B6} J. Alves and V. Pinheiro. \textit{Topological structure of partially hyperbolic sets with positive volume}, Trans. Amer. Math. Soc. 360 (2008), no. 10, 5551-5569. 

%\bibitem{B9} C. Bonatti, \textit{$C^1$-Generic Dynamics: Tame and Wild Behaviour}, Proceedings of the International Congress of Mathematicians, Vol. III (Beijing, 2002), 265-277, Higher Ed. Press, Beijing, 2002. 


\bibitem{B34} C. Bonatti and S. Crovisier, \emph{Recurrence et genéricité}, Invent. Math. 158 (2004), no. 1, 33-104. 


%\bibitem{B20} C. Bonatti, S. Crovisier, N. Gourmelon, and R. Potrie, \emph{Tame dynamics and robust transitivity}, arxiv preprint, arXiv:1112.1002v1, 2011.



%\bibitem{B14} C. Bonatti and L. Díaz, \emph{Connexions hétéroclines et généricité d'une infinité de puits et de sources}, Ann. Sci. École Norm. Sup. (4) 32 (1999), no. 1, 135-150. 

 
%\bibitem{B19} C. Bonatti e L. Díaz, \emph{Nonhyperbolic transitive diffeomorphisms}, Ann. of Math. no. 2,  143: 357-396, 1996.   
 
 
 
%\bibitem{B23} C. Bonatti, L. Díaz, E. Pujals, \emph{A $C^1$ generic dichotomy for diffeomorphisms: weak forms of hyperbolicity infinitely many sinks or sources}, Ann. of Math., 158: 355-418, 2003.
 
%\bibitem{B7} C. Bonatti, L. Díaz, E. Pujals, and J. Rocha. \textit{Robustly transitive sets and heterodimentional cycles}, Astérisque No. 286 (2003), xix, 187-222. 


\bibitem{B2} C. Bonatti, L. Díaz, and R. Ures. \textit{Minimality of Strong Stable and Unstable Foliations for Partially Hyperbolic Diffeomorphisms}, J. Inst. Math. Jussieu 1 (2002), no. 4, 513-541. 

\bibitem{B3} C. Bonatti, L. Díaz, and M. Viana. \textit{Dynamics Beyond Uniform Hyperbolicity}, Springer Verlag, 2007.

\bibitem{B12} C. Bonatti, S. Gan, and L. Wen, \emph{On the existence of non-trivial homoclinic classes}, Ergod. Th. \& Dynam. Sys 27: 1473-1508, 2007.     
 

\bibitem{B17} C.  Bonatti and M. Viana, \emph{SRB measures for partially hyperbolic systems whose central direction is mostly contracting}, Israel J. Math. 115 (2000), 157-193. 


\bibitem{B31} R. Bowen, \emph{A horseshoe with positive measure}, Invent. Math. Vol. 29, no. 3, 203-204, 1975.


\bibitem{B36} R. Bowen, \emph{Equilibrium States and the Ergodic Theory of Anosov Diffeomorphisms}, Lecture Notes on Mathematics, 470. Springer, Berlin , 1945.
%\bibitem{B33} M. Brin, \emph{Topological transitivity of a certain class of dynamical systems, and flows of frame on manifolds of negative curvature}, Funct. anal. Appl. 9: 9-19, 1975.


\bibitem{B40} M. Brin , \emph{ Topological transitivity  of a certain class of dynamical systems, and flows of frames on manifolds of negative curvature}, funkcional. Anal. i Prilozen 9 (1975), no 1, 9-19.



%\bibitem{B8} C. Carballo, C. Morales, and M. J. Pacifico. \textit{Homoclinic classes for $C^1$-generic vector fields}, Ergodic Theory Dynam. Systems 23 (2003), no. 2, 403-415. .


\bibitem{B30} C. Conley, \emph{Isolated invariant sets and the Morse index}, CBMS Regional Conference Series in Mathematics, 38. Amer. Math. Soc., Providence, R.I., 1978.


%\bibitem{B32} D. Dolgopyat and A. Wilkinson, \emph{Stable aceccibility is $C^1$-dense}, Asterisque 287: 33-60, 2003.


\bibitem{B35} T. Fisher, \emph{Hyperbolic sets with non-empty interior}, Discrete Contin. Dyn. Syst. 15(2) (2006), 433-446.

\bibitem{B10} F. Hertz, M. Hertz, and R. Ures, \emph{Some results on the integrability of the center bundle for partially hyperbolic diffeomorphisms}, Partially hyperbolic dynamics, laminations, and Teichmüller flow, 103-109, Fields Inst. Commun., 51, Amer. Math. Soc., Providence, RI, 2007. 

\bibitem{p1} F. Nobili, \emph{ Minimality of one invariant  foliation for partially hyperbolic attractors}, Nonlinearity 28 (2015), 1897–1918. 

%\bibitem{B11} N. Gourmelon \emph{Adapted metrics for dominated splittings}, Ergodic Theory Dynam. Systems 27 (2007), no. 6, 1839-1849. 



%\bibitem{B1} M. Hirsch, C. Pugh, and M. shub. \textit{Invariant Manifolds}, volume 583 of Lect. Notes in Math. Spinger Verlag, 1977.


 
\bibitem{B16} R. Mañe,  \emph{An ergodic closing lemma}, Ann. of Math. (2) 116 (1982), no. 3, 503-540. 


%\bibitem{B22} J. Palis, \emph{A global view of Dynamics and a conjecture on the denseness of finitude of attractors}, Astérisque No. 261 (2000), xiii–xiv, 335-347.  


%\bibitem{B25} J. Plante \emph{Anosov flows}, Amer. J. Math. 94 (1972), 729-754.


\bibitem{B26} R. Potrie \emph{Generic bi-Lyapunov stable homoclinic classes}, Nonlinearity 23 (2010), no. 7, 1631-1649. 


%\bibitem{B13} C. Pugh, \emph{An improved Closing Lemma and a general density theorem} , Amer. J. Math. 89 1967 1010-1021. 

%\bibitem{B15} M. Viana, \emph{Lecture notes on attractors and physical measures}, Instituto de Matemática y Ciencias Afines, IMCA, Lima, 1999.
 
 
%\bibitem{B18} M. Shub, \emph{Topologically transitive diffeomorphisms in $\mathbb{T}^4$}, Springer Verlag , 1971.              

\bibitem{B24} S. Smale, \emph{Diffeomorphisms with many periodic points}, Bull. Am. Math. Soc., 1967. 73: 747-817.

\end{thebibliography}
\end{document}